\documentclass[a4paper, 11pt]{amsart}
\usepackage{a4wide}
\usepackage[english]{babel} 
\usepackage{amsmath, amssymb,amsthm} 
\usepackage[all,cmtip]{xy} 
\usepackage{tikz}
\usepackage[utf8]{inputenc}
\usepackage{tikz-cd}
\usepackage{xcolor}
\usepackage[normalem]{ulem}
\usepackage[mathscr]{euscript}
\usepackage{pifont}
\usepackage{hyperref}
\usepackage[capitalize,noabbrev]{cleveref}
\usepackage{dsfont}

\newcounter{commentcounter}

\newtheorem{Prop}{Proposition}

\newtheorem{thm}[Prop]{Theorem}

\theoremstyle{definition}

\newtheorem{rem}[Prop]{Remark}

\newcommand{\CP}{\mathbb{CP}}
\newcommand{\RP}{\mathbb{RP}}
\newcommand{\Z}{\mathbb{Z}}
\DeclareMathOperator{\km}{km}

\begin{document}

\title{Topologically flat embedded 2-spheres in specific simply connected 4-manifolds}

\author{Daniel Kasprowski}
\address{Rheinische Friedrich-Wilhelms-Universit\"at Bonn, Mathematisches Institut,\newline\indent Endenicher Allee 60, 53115 Bonn, Germany}
\email{kasprowski@uni-bonn.de}
\author{Peter Lambert-Cole}
\address{School of Mathematics, Georgia Institute of Technology}
\email{plc@math.gatech.edu}
\author{Markus Land}
\address{Fakultät für Mathematik, Universität Regensburg, 93040 Regensburg, Germany}
\email{markus.land@mathematik.uni-regensburg.de}
\author{Ana G. Lecuona}
\address{School of Mathematics \& Statistics, University of Glasgow, University Place, \newline\indent Glasgow G12 8QQ}
\email{Ana.Lecuona@glasgow.ac.uk}

\begin{abstract}
	In this note we study whether specific elements in the second homology of specific simply connected closed $4$-manifolds can be represented by smooth or topologically flat embedded spheres. 
\end{abstract}

\maketitle
\section*{Introduction}
Let $X$ be a simply connected closed $4$-manifold and consider an element in its second homology group. It is well known that any such element can be represented by an embedded closed oriented surface. Finding the minimal genus among all such representing surfaces is an interesting task; see \cite{Lawson97} for a survey on this topic and \cite{FV14,Nouh14,Nagel16} for some recent results. The aim of this note is to discuss specific examples in which it is possible (or not) to push the genus down to be zero, i.e.\ where it is possible (or not) to represent specific elements by embedded spheres.
Notice that as $X$ is simply connected, any element in the second homology can be represented by a sphere and since every topological $4$-manifold is smoothable away from a point \cite[Corollary~2.2.3]{Quinn82}, we can assume the sphere to be regularly immersed. Hence the above question is equivalent to asking whether this regularly immersed sphere is homotopic to an embedded one.

To be more precise, in this note we consider the manifolds $M=8\CP^2 \# \overline{\CP^2}, M'=8\CP^2 \# \overline{\star\CP^2}$ and $\CP^2\#M$. The groups $H_2(M;\Z)$, $H_2(M';\Z)$ and $H_2(\CP^2\#M;\Z)$ will be considered with their `evident' bases; for example, for $H_2(M;\Z)$ the basis consists of 8 spheres of self-intersection 1, denoted $e_{1},\dots,e_{8}$, and a last sphere $e_{9}$ of self-intersection $-1$. Within these groups and with respect to the evident bases, we will be interested in the elements $x=(1,\ldots,1,3)\in H_2(M;\Z)$, $x'=(1,\ldots,1,3)\in H_2(M';\Z)$ and $(0,x)\in H_2(\CP^2\#M;\Z)$. The aim of this note is to show that $x$ cannot be represented by a topologically flat embedded sphere while the elements $x'$ and $(0,x)$ can be represented in such a way. At first we use the Kirby-Siebenmann invariant as an obstruction for $x$ and Freedman's classification of simply connected manifolds to prove the statement for $x'$ and $(0,x)$. In the second part of this paper, we reprove these statements using the Kervaire-Milnor invariant.

\section*{Existence of smooth or topologically flat embedded representatives}
 The intersection form of $M=8\CP^2 \# \overline{\CP^2}$ is given by 
\[\lambda_M = 8\langle 1 \rangle \oplus \langle -1 \rangle\]
so we observe the following:
\begin{enumerate}
\item $x=(1,\ldots,1,3)$ is a characteristic element for $\lambda_M$, i.e.\ $x\cdot y = y\cdot y$ mod 2 for all $y\in H_2(M;\Z)$, and $x\cdot x = -1$.
\item the orthogonal complement $\langle x \rangle^\perp$ of $\langle x\rangle$ is isomorphic to $E_8$: In fact the following elements give a basis of $\langle x \rangle^\perp$ whose representing matrix is the $E_8$-matrix: For $1\leq j\leq 7$ take $f_j = e_{j+1}-e_j$, and let $f_8 = e_9-e_6-e_7 - e_8$.
\end{enumerate}

\begin{rem}
	\label{rem:smoothobstruction}
	The element $x$ cannot be represented by a smoothly embedded sphere. Arguing by means of contradiction, we assume that $x$ can be represented by a smoothly embedded sphere. Then the normal bundle of such an embedding is a complex line bundle with Euler class $-1$. Its associated disk bundle is thus diffeomorphic to $\overline{\CP}^2$ with an open disk removed and its sphere bundle is $S^3$. Removing the interior of a tubular neighbourhood of the embedded $S^2$ one obtains a manifold with boundary $S^3$ to which we can glue $D^4$. This construction is called a ``blow-down'' since it is the inverse to a ``blow-up'', i.e.\ taking connected sum with $\overline{\CP}^2$. The intersection form of the resulting smooth (and simply connected) manifold $X$ is isometric to the orthogonal complement of $\langle x\rangle$, hence isometric to $E_8$. In particular, its intersection form is even, so that $X$ is a smooth spin manifold. By Rokhlin's theorem the signature cannot be $8$ and we obtained a contradiction.
\end{rem}
%

In contrast to the smooth case considered in the previous remark, a topological spin 4-manifold with signature $8$ exists by the work of Freedman. So one is led to ask whether $x$ can be represented by a topologically flat embedding of $S^2$. We answer this question to the negative.
\begin{thm}
	\label{thm:1}
	The element $x$ cannot be represented by a topologically flat embedding $S^2\to M$.
\end{thm}
\begin{proof}[Proof of \cref{thm:1}]
As in the smooth case, a topologically flat embedding would give rise to a simply connected topological manifold with intersection form $E_8$. This manifold exists and is unique up to homeomorphism by Freedman's classification of simply connected topological $4$-manifolds (\cite[Theorem~1.5]{freedman}). It is denoted by $E8$. More precisely, the topologically flat embedding would yield a homeomorphism $M\cong E8\# \overline{\CP^2}$. 
Recall that the Kirby--Siebenmann invariant $KS$ is an obstruction for topological manifolds to admit a smooth (in fact PL) structure. It is a bordism invariant of $4$-manifolds,  hence we have
\[KS(M\#N)=KS(M)+KS(N),\]
 for all $4$-manifolds $M$ and $N$.
As $E8$ is spin, we have $KS(E8)\equiv\sigma(E8)/8\mod2$ by \cite[Theorem 13.1]{KS}, where $\sigma(-)$ denotes the signature. As $\CP^2$ is smooth, we have $KS(\overline{\CP^2})=0$. Thus we find $KS(M)=0$ but $KS(E8\# \CP^2)=1$ so that these manifolds are not homeomorphic. It follows that $x$ cannot be represented by a topologically flat embedding of $S^2$.
\end{proof}

Let us consider the following variant, namely the manifold $M' = 8\CP^2 \# \overline{\star\CP^2}$,  where $\star \CP^2$ is a fake $\CP^2$, i.e.\ a manifold homotopy equivalent to $\CP^2$, but with non-trivial Kirby--Siebenmann invariant. The existence of such a manifold again makes use of Freedman's theorem. One nice way to construct $\star \CP^2$, see \cite[pg. 370]{freedman} where it is called the Chern manifold, is to consider the Poincar\'e 3-sphere, which by Freedman's theorem bounds a unique contractible 4-manifold. The trace of a $+1$-framed surgery on the trefoil knot produces a 4-manifold with boundary given by the Poincar\'e 3-sphere. Glueing together these two manifolds along their common boundary produces a simply connected 4-manifold with 
intersection form $\langle 1 \rangle$, so this manifold is homotopy equivalent to $\CP^2$. However, one can check that its Kirby--Siebenmann invariant is non-trivial.

The intersection form of $M'$ is thus also given by 
\[\lambda_{M'} = 8\langle 1 \rangle \oplus \langle -1 \rangle\]
and we consider again the element $x' = (1,\dots,1,3) \in H_2(M';\Z)$.
\begin{thm}
	\label{thm:2}
	The element $x'$ can be represented by a topologically flat embedding $S^2\to M'$.
\end{thm}
\begin{proof}
Since there is an isomorphism of forms $8\langle1\rangle\oplus\langle-1\rangle\cong E8\oplus\langle-1\rangle$ sending $x'=(1,\ldots,1,3)$ to $(0,1)$ and $KS(8\CP^2\#\star\overline{\CP}^2)=KS(E8\#\overline{\CP}^2)=1$, there is an homeomorphism $8\CP^2\#\star\overline{\CP}^2\cong E8\#\overline{\CP}^2$ which sends $x'$ to a generator of $H_2(\overline{\CP}^2;\Z)$ by Freedman's classification of simply connected topological manifolds (\cite[Theorem~1.5 and its addendum]{freedman}). The theorem now follows from the fact that the generator of $H_2(\overline{\CP}^2;\Z)$ can be represented by a smoothly embedded sphere.
\end{proof}

Finally, we consider $\CP^2\# M = 9\CP^2\#\overline{\CP^2}$ and the element $(0,x)=(0,1,\ldots,1,3)\in H_2(\CP^2\# M;\Z)$.
\begin{thm}
	\label{thm:3}
	The element $(0,x)$ can be represented by a topologically flat embedding $S^2\to \CP^2\#M$, but not by a smooth embedding.
\end{thm}
\begin{proof}
	Again by Freedman's classification of simply connected topological manifolds, there is a homeomorphism $\CP^2\#M\cong \star\CP^2\# M'$ sending $(0,x)$ to $(0,x')$ and we know by \cref{thm:2} that $(0,x')$ can be represented by a topologically flat embedding.
	
	As in \cref{rem:smoothobstruction}, if $x$ were represented by a smooth embedding, there would be a smooth manifold with intersection form $E_8\oplus\langle1\rangle$. Since $E_8\oplus\langle1\rangle$ is definite but not diagonalizable, no compact, smooth, simply connected, orientable $4$-manifold with this intersection form exists by Donaldson's theorem (\cite[Theorem~1]{donaldson}).
\end{proof}
\section*{The Kervaire--Milnor invariant}
The topological part of the above theorems can also be shown using the Kervaire--Milnor invariant $\km$, introduced by Freedman and Quinn \cite[Definition~10.8A]{FQ}. For a complete treatment of the Kervaire--Milnor invariant we refer to \cite[Section~9]{PRT}. For an immersed $2$-sphere $\iota$ with an algebraically dual sphere $g$, i.e.\ $\lambda(\iota,g)=0$, in a simply connected closed $4$-manifold $M$ the Kervaire--Milnor invariant takes values in $\Z/2$ if $\iota$ is $s$-characteristic and it lives in the trivial group if $\iota$ is not $s$-characteristic. Here, an immersed $2$-sphere $\iota$ is called $s$-characteristic, if for every other immersed $2$-sphere $\iota'$ one has $\lambda_M(\iota,\iota')\equiv \iota'\cdot\iota'\mod 2$. 

One can describe $\km(\iota)$ as follows: Assume that the algebraic self-intersection number $\mu(\iota)$ vanishes (this can always be achieved by introducing local kinks, which only change the Euler number of the normal bundle by an even number). In this case, the geometric self-intersection points of $\iota$ can be paired up in couples with canceling signs. Therefore, one can choose a \emph{framed} Whitney disc for each pair of self-intersections, and arrange that all the boundary arcs are disjoint. Then one counts the mod 2 number of intersection points of the interior of the Whitney disks with $\iota$. The number obtained in this fashion then does not depend on the particular choice of Whitney disks, and the particular choices of changing $\iota$ to have algebraic self-intersection number $0$.

Note that in general the Kervaire--Milnor invariant lives in the trivial group if $\iota$ is not $r$-characteristic, where $\iota$ is called $r$-characteristic if it is $s$-characteristic and for every immersion $\iota'\colon \RP^2\to M$ we have for the $\Z/2$-intersection form that $\iota\cdot \iota'=\iota'\cdot\iota'$. However in a simply connected $4$-manifold every immersion $\RP^2\to M$ factors up to homotopy through $\RP^2/\RP^1\simeq S^2$ and hence every $s$-characteristic sphere is $r$-characteristic.

\begin{thm}[{\cite[pp.~1310-1311]{stong}}]
	\label{thm:embeddings}
	An immersed $2$-sphere with an algebraically dual sphere is homotopic to a topologically flat embedded sphere if and only if it has trivial Kervaire--Milnor invariant. 
\end{thm}

\subsection*{Reproving \cref{thm:1}} 
Recall that we are considering the manifold $M=8\CP^2\#\overline{\CP^2}$ and the element $x=(1,\ldots,1,3)\in H_2(M;\Z)$. The element $x$ intersects the canonical 2-sphere representing $(1,0\dots,0) \in H_2(M;\Z)$ in a single point and hence has an algebraically dual sphere. Furthermore, $x$ is characteristic and hence $s$-characteristic. This implies that $\km(x)$ lives in $\Z/2$.

We can pick embedded spheres representing $1\in H_2(\CP^2;\Z)$ and an immersed sphere $y$ in $\overline{\CP^2}$ representing $3\in H_2(\overline{\CP^2};\Z)$. We can add local kinks to $y$ and pick Whitney disks inside $\overline{\CP^2}$ for $y$ to compute $\km(y)$. The element $x$ can be represented by taking a connected sum of the embedded spheres representing the generator in $H_2(\CP^2;\Z)$ and $y$. We can take the connected sum in $M$ avoiding the Whitney disks chosen for $y$. It follows that $\km(x)=\km(y)$. By \cite[p.~1313]{stong}, we have the formula 
\begin{equation}
\label{eq:stong}
\km(\iota)\equiv (\iota\cdot\iota-\sigma(M))/8 + KS(M) \mod 2
\end{equation}
for an $s$-characteristic sphere $\iota$ in a simply connected $4$-manifold. Thus
\[\km(y)\equiv (y\cdot y-\sigma(\overline{\CP^2}))/8+KS(\overline{\CP^2})=(-9-(-1))/8+0=-1.\]
By \cref{thm:embeddings}, $\km(x)=\km(y)=1$ implies that $x$ cannot be represented by a topologically flat embedding.

\subsection*{Reproving \cref{thm:2}}
Recall that we are considering the manifold $M'=8\CP^2\#\overline{\star\CP^2}$ and the element $x'=(1,\ldots,1,3)\in H_2(M';\Z)$. It has the same algebraically dual sphere as $x$. For $x'$ we can consider the homeomorphism $M'\cong \star\CP^2\#7\CP^2\#\overline{\CP^2}$. As for $x$, we see that $\km(x')=\km(y')+\km(y)$ where $y'$ represents the generator of $H_2(\star\CP^2;\Z)$. Using \eqref{eq:stong}, we have
\[\km(y')\equiv (y'\cdot y'-\sigma(\star \CP^2))/8+KS(\star \CP^2)= (1-1)/8+1=1.\]
Hence $\km(x')=1+1=0$. By \cref{thm:embeddings}, $x'$ can be represented by a topologically flat embedding.

\subsection*{Reproving \cref{thm:3}}
Recall that we are considering the manifold $\CP^2\#M=9\CP^2\#\overline{\CP^2}$ and the element $(0,x)=(0,1,\ldots,1,3)\in H_2(\CP^2\#M;\Z)$.
Consider the element $z:=(1,(1,0,\ldots,0))\in\CP^2\#M$. Then $\lambda(z,z)=2$ and $\lambda(z,(0,x))=1$. Hence $(0,x)$ is not $s$-characteristic. Therefore $\km((0,x))$ lives in the trivial group and hence is itself trivial. By \cref{thm:embeddings}, $(0,x)$ can be represented by a topologically flat embedding.
	
\section*{Acknowledgments}
The authors thank the Matrix Institute for hospitality during the workshop ``Topology of Manifolds: interactions between high and low dimensions'' where the question how to prove \cref{thm:1} came up during a discussion session. Most of the work on this article was done during this workshop. The authors also thank Peter Teichner for helpful discussions in particular about the Kervaire--Milnor invariant.

The first author was supported by the Deutsche Forschungsgemeinschaft (DFG, German Research Foundation) under Germany's Excellence Strategy - GZ 2047/1, Projekt-ID 390685813. The third author was supported by the SFB 1085 ``Higher Invariants'' in Regensburg.
\bibliographystyle{amsalpha}
\bibliography{mybib}
\end{document}